\documentclass[11pt,onecolumn]{IEEEtran}
\usepackage{amsmath,amssymb,euscript ,yfonts,psfrag,latexsym,dsfont,graphicx,bbm,color,amstext,wasysym,subfig,parskip}
\usepackage{flushend}
\graphicspath{{./},{./figures/}}

\begin{document}
\newtheorem{thm}{Theorem}
\newtheorem{cor}[thm]{Corollary}
\newtheorem{conj}[thm]{Conjecture}
\newtheorem{lemma}[thm]{Lemma}
\newtheorem{prop}{Proposition}
\newtheorem{problem}[thm]{Problem}
\newtheorem{remark}[thm]{Remark}
\newtheorem{defn}[thm]{Definition}
\newtheorem{ex}[thm]{Example}

\newcommand{\mR}{{\mathbb R}}
\newcommand{\mD}{{\mathbb D}}
\newcommand{\mE}{{\mathbb E}}  
\newcommand{\E}{{\mathbb E}}
\newcommand{\cN}{{\mathcal N}}
\newcommand{\cR}{{\mathcal R}}
\newcommand{\cS}{{\mathcal S}}
\newcommand{\cC}{{\mathcal C}}
\newcommand{\cU}{{\mathcal U}}
\newcommand{\diag}{\operatorname{diag}}
\newcommand{\tr}{\operatorname{trace}}
\newcommand{\f}{{\mathfrak f}}
\newcommand{\g}{{\mathfrak g}}
\newcommand{\range}{\cR}  
\newcommand{\trace}{\operatorname{trace}}
\newcommand{\argmin}{\operatorname{argmin}}

\newcommand{\ignore}[1]{}
\IEEEoverridecommandlockouts
\overrideIEEEmargins

\def\spacingset#1{\def\baselinestretch{#1}\small\normalsize}
\setlength{\parskip}{10pt}
\setlength{\parindent}{20pt}
\spacingset{.95}

\newcommand{\mike}{\color{magenta}}
\newcommand{\mmike}{\color{blue}}
\definecolor{grey}{rgb}{0.6,0.6,0.6}
\definecolor{lightgray}{rgb}{0.97,.99,0.99}

\title{Steering state statistics with output feedback}

\author{Yongxin Chen, Tryphon T. Georgiou and Michele Pavon\thanks{Y.\ Chen and T.T.\ Georgiou are with the Department of Electrical and Computer Engineering,
University of Minnesota, Minneapolis, Minnesota MN 55455, USA; {email: \{chen2468,tryphon\}@umn.edu}}
\thanks{M.\ Pavon is with the Dipartimento di Matematica,
Universit\`a di Padova, via Trieste 63, 35121 Padova, Italy; {email: pavon@math.unipd.it}}\thanks{The research was supported in part by the NSF under Grant 1027696 and the AFOSR under Grants FA9550-12-1-0319 and FA9550-15-1-0045}}

\maketitle
{\begin{abstract}
Consider a linear stochastic system
whose initial state is a random vector with a specified Gaussian distribution.
Such a distribution may represent a collection of particles abiding by the specified
system dynamics. In recent publications, we have shown that, provided the system is controllable,
it is always possible to steer the state covariance to any specified terminal Gaussian distribution
using {\em state feedback}.
The purpose of the present work is to show that, in the case where only partial state observation is available,
a necessary and sufficient condition for being able to steer the system to a specified terminal Gaussian distribution for the state vector is that the terminal state covariance be greater (in the positive-definite sense) than the error covariance of a corresponding Kalman filter.
\end{abstract}}


\section{Introduction}
The classical paradigm of stochastic optimal control is to regulate the response of system in such a way so as to minimize a prescribed performance index. The quality of the response in terms of tracking reference signals and/or reaching a prescribed destination point is encoded in the performance index which penalizes deviation from desirable mean-value response. The effect of state uncertainty at the starting point, and of stochastic disturbances and measurement noise, is that sample paths of state and output processes incur a certain amount of spread. The role of the performance index is precisely to limit this spread, indirectly, as a result of the optimal strategy that keeps the cost low and hence curtails deviation from the desired mean response.

The viewpoint presented here is based on recent work by the authors \cite{CheGeoPav14a,CheGeoPav14b} and departs substantially from this classical recipe and aims to specify directly the spread of the state-vector.  Thus, in this work, we first considered the question of whether specific state-distributions are attainable over a finite or infinite time-interval through (noise-free) state feedback. In the present, we address for the first time the control of the state-distribution, over a finite or infinite interval, via output feedback in the presence of measurement noise. It turns out that, our ability to steer the state-distribution, as compared to what is possible by noise-free state feedback, is only limited by an inequality of admissible state-covariances to exceed the error covariance of a corresponding Kalman filter.

For motivation and background we refer to \cite{CheGeoPav14a,CheGeoPav14b} as well as to the largely expository  paper \cite{CheGeoPav15c} which has also been submitted in these proceedings (CDC 2015).


\section{Finite horizon steering}\label{sec:finitehorizon}
Consider the linear time-invariant (LTI) system
    \begin{subequations}\label{eq:dynamics}
    \begin{eqnarray}
    dx(t)&=&Ax(t)dt+Bu(t)dt+B_1dw(t)\\
    dy(t)&=&Cx(t)dt+Ddv(t)
    \end{eqnarray}
    \end{subequations}
where
\[
(A,B,B_1,C,D)\in\mR^{n\times n}\times\mR^{n\times m}\times\mR^{n\times m_1}\times\mR^{p\times n}\times\mR^{p\times p},
\]
$(A,B)$ is controllable, $(A,C)$ is observable, $D$ is invertible, and $x,y,u,w,v$ represent the state, output, control input, process noise, and measurement noise respectively.
We assume that $w$ and $v$ are both standard Wiener processes and independent of each other.
We assume that at $t=0$ the state, $x(0)$, is Gaussian with mean equal to zero and covariance $\Sigma_0>0$, i.e., having probability density
\begin{equation}\label{initial}\rho_0(x)=(2\pi)^{-n/2}\det (\Sigma_0)^{-1/2}\exp\left(-\frac{1}{2}x'\Sigma_0^{-1}x\right).
\end{equation}
The assumption on having zero-mean is made only for simplicity of the exposition and can be easily removed.

Our goal is to steer the state distribution using dynamic output feedback
to a ``target'' Gaussian end-point distribution
 \begin{equation}\label{final}
 \rho_T(x)=(2\pi)^{-n/2}\det (\Sigma_T)^{-1/2}\exp\left(-\frac{1}{2}x'\Sigma_T^{-1}x\right),
\end{equation}
for the state vector, with $\Sigma_T$ a given symmetric and  positive definite $n\times n$ matrix.

\subsection{Feasibility conditions}
We first determine conditions on the terminal state covariance $\Sigma_T$ that permit the existence of control that steers the stochastic system to the corresponding end-point state density \eqref{final}.

Consider state estimates provided by the Kalman filter
    \begin{equation}\label{eq:kalmanfilter}
        d\hat{x}(t)=A\hat{x}(t)dt+Bu(t)dt+L(t)(dy-C\hat{x}dt),
    \end{equation}
with (optimal) gain
    \[
        L(t)=P(t)C'(DD')^{-1}
    \]
and $P(t)$ the state error-covariance obtained by
solving the differential Riccati equation
    \begin{equation}\label{eq:Riccati}
        \dot{P}(t)=AP(t)+P(t)A'+B_1B_1'-P(t)C'(DD')^{-1}CP(t)
    \end{equation}
with initial condition $P(0)=\Sigma_0$. As usual, we denote by $\tilde{x}(t)=x(t)-\hat{x}(t)$ the estimation error, which satisfies
    \[
        d\tilde{x}(t)=(A-L(t)C)\tilde{x}(t)dt+B_1dw(t)-L(t)Ddv(t)
    \]
and is orthogonal to $\hat x(t)$, i.e., $\mE(\hat x(t)\tilde x(t)')=0$.
It follows that
    \begin{eqnarray}\label{eq:optimalityofKF}
        \Sigma(T)&:=& \mE(x(T)x(T)')\\\nonumber
        &=&\mE((\hat{x}(T)+\tilde{x}(T))(\hat{x}(T)'+\tilde{x}(T)'))\\\nonumber
        &=&\mE(\tilde{x}(T)\tilde{x}(T)')+\mE(\hat{x}(T)\hat{x}(T)')\\\nonumber
        &=&P(T)+\mE(\hat{x}(T)\hat{x}(T)')\ge P(T).
    \end{eqnarray}
Therefore,
\begin{equation}\label{eq:necessary}
\Sigma_T\geq P(T)
\end{equation}
is a {\em necessary condition} for a terminal state covariance to be ``reachable'' through suitable steering of the system dynamics. Our first result states that the {\em strict inequality is $\Sigma_T>P(T)$ is in fact sufficient}. This relies on \cite[Theorem 3]{CheGeoPav14b} which establishes a ``controllability'' result for a matrix differential Lyapunov equation. We note that the above argument on the necessity of \eqref{eq:necessary} does not assume any particular form for the functional dependence of the control input $u$ on the output $y$. Yet, in the proof of the theorem below it is seen that, under the slightly stronger condition $\Sigma_T>P(T)$, a control input of the form $u=-K\hat x$ is sufficient to ensure the terminal distribution of $x(T)$.

\begin{thm}\label{thm:thm1}
Given the stochastic linear system \eqref{eq:dynamics} with distribution for the state vector at $t=0$ specified by
\eqref{initial}, and given
\begin{equation}\label{eq:sufficient}
\Sigma_T>P(T)
\end{equation}
 where $P(t)$ satisfies the Riccati equation \eqref{eq:Riccati} with initial condition $P(0)=\Sigma_0$, there exists a control process $u(t)$, adapted to the output process $y(t)$, such that the distribution of the state vector at $t=T$ is given by the density in \eqref{final}.
\end{thm}

\begin{proof}
Consider a control $u=-K(t)\hat{x}(t)$ where $\hat{x}(t)$ is the state of the Kalman filter and $K(t)$ a time-varying gain matrix. The state covariance of the combined state + estimation error
system
    \begin{eqnarray*}
        \left[\begin{array}{c}dx\\d\tilde{x}\end{array}\right]&=&
        \left[\begin{array}{cc}A-BK & BK\\ 0 & A-LC\end{array}\right]
        \left[\begin{array}{c}x\\\tilde{x}\end{array}\right]dt\\&&+
        \left[\begin{array}{c}B_1dw\\B_1dw-LDdv\end{array}\right],
    \end{eqnarray*}
satisfies the matrix differential Lyapunov equation
    \begin{eqnarray*}
    \left[\begin{array}{cc}\dot{\Sigma} & \dot{P}\\\dot{P} & \dot{P}\end{array}\right]&=&
    \left[\begin{array}{cc}A-BK & BK\\ 0 & A-LC\end{array}\right]
    \left[\begin{array}{cc}\Sigma & P\\P & P\end{array}\right]\\&&+
    \left[\begin{array}{cc}\Sigma & P\\P & P\end{array}\right]
    \left[\begin{array}{cc}A-BK & BK\\ 0 & A-LC\end{array}\right]'\\
    &&+\left[\begin{array}{cc}B_1B_1' & B_1B_1'\\ B_1B_1' & B_1B_1'+LDD'L'\end{array}\right].
    \end{eqnarray*}
Denote by $\hat{\Sigma}:=\Sigma-P$ the error covariance of the Kalman filter state. From the above it readily follows that
    \begin{equation}\label{eq:covdifference}
    \dot{\hat\Sigma}=(A-BK)\hat{\Sigma}+\hat{\Sigma}(A-BK)'+LDD'L' .
    \end{equation}
Since $\Sigma(t)=\hat{\Sigma}(t)+P(t)$ for all $t$, it suffices to steer $\hat\Sigma(t)$ to a terminal value $\hat{\Sigma}(T)=\Sigma_T-P(T)>0$ with a suitable choice of $K(t)$
over $[0,T]$.

The claim of the theorem now basically follows from \cite[Theorem 3]{CheGeoPav14b} which states that
a differential Lyapunov equation
\begin{equation}
\nonumber
\dot{Q}=AQ+QA'+BU(t)'+U(t)B'
\end{equation}
is controllable, i.e., $Q(t)$ can be steered by a proper choice of $U(t)$ between any two conditions at $t=0$ and $t=T$, if and only if the system \eqref{eq:dynamics} is controllable, and moreover, the path $Q(t)$ for $t\in[0,T]$ can remain within the cone of positive definite matrices provided the boundary conditions $Q(0)$ and $Q(T)$ are. The corresponding value for $K(t)$ is $U(t)'Q(t)^{-1}$. The only technical issue we need to address is that the conditions in \cite[Theorem 3]{CheGeoPav14b} require that the initial covariance be positive definite, while here, the initial condition for \eqref{eq:covdifference} is
 $\hat\Sigma(0)=0$. To this end, we consider the choice $K\equiv 0$ over a short window of time, $[0,\epsilon)$. As we explain below, the conditions of the theorem will be fulfilled at $t=\epsilon$, and thence we apply \cite[Theorem 3]{CheGeoPav14b}.

We claim that the solution to \eqref{eq:covdifference} with $K\equiv 0$, namely
    \begin{equation}\label{eq:zerogain}
        \dot{\hat\Sigma}=A\hat{\Sigma}+\hat{\Sigma}A'+LDD'L'  \mbox{ with }\hat\Sigma(0)=0,
    \end{equation}
satisfies that $\hat{\Sigma}(t)>0$ for any $t>0$.
To see that this is true, first note that $\hat{\Sigma}:=\Sigma-P$ where $P$ satisfies \eqref{eq:Riccati} and $\Sigma$ satisfies
    \begin{equation}\label{eq:covariance}
        \dot{\Sigma}=A\Sigma+\Sigma A'+B_1B_1'.
    \end{equation}
Hence, both $\Sigma(\cdot)$ and $P(\cdot)$ are continuously differentiable functions with the same initial value $\Sigma(0)=P(0)=\Sigma_0$ and therefore, $\hat\Sigma(\epsilon)=\Sigma(\epsilon)-P(\epsilon)$ is of order $O(\epsilon)$ for small $\epsilon>0$. Rewrite \eqref{eq:covariance} and \eqref{eq:Riccati} as
    \[
        \dot{\Sigma}^{-1}=-A'\Sigma^{-1}-\Sigma^{-1}A-\Sigma^{-1}B_1B_1'\Sigma^{-1},
    \]
    \[
        \dot{P}^{-1}=-A'P^{-1}-P^{-1}A-P^{-1}B_1B_1'P^{-1}+C'(DD')^{-1}C.
    \]
It follows that
    \begin{align}\nonumber
        e^{A't}(P(t)^{-1}-\Sigma(t)^{-1})e^{At} &= \int_0^t e^{A'\tau}C'(DD')^{-1}Ce^{A\tau}d\tau\\
        &\hspace{-1cm}-\int_0^t M(\tau) d\tau \label{eq:Mtau}
    \end{align}
    where
     \begin{align*}
       M(\tau)&= e^{A'\tau}\left(P(\tau)^{-1}B_1B_1'P(\tau)^{-1}\right.\\
       &\left.\phantom{xxxxxx}-\Sigma(\tau)^{-1}B_1B_1'\Sigma(\tau)^{-1}\right)e^{A\tau}
    \end{align*}
    is differentiable and satisfies $M(0)=0$.
Thus, the second term on the right hand side of equation \eqref{eq:Mtau} is of order $O(t^2)$ for small $t$
while the first term,
\[
\int_0^te^{A'\tau}C'(DD')^{-1}Ce^{A\tau}d\tau
\]
is of order $O(t)$. Moreover, this first term is strictly positive for any $t>0$ since
$(A, C)$ is observable.
Therefore,
    \[
        e^{A't}(P(t)^{-1}-\Sigma(t)^{-1})e^{At}>0
    \]
for $t>0$ and small enough. We conclude that $\Sigma(t)>P(t)$, that is $\hat{\Sigma}(t)>0$, for small $t$. The fact that $\hat\Sigma(t)>0$ for larger $t$ as well readily follows from \eqref{eq:zerogain}.
\end{proof}

\subsection{Sufficient conditions for optimality}

So far, we have established
that provided $\Sigma_T>P(T)$ as in Theorem \ref{thm:thm1},
it is possible to steer \eqref{eq:dynamics} from the initial probability density $\rho_0$ to the ``target'' final probability density $\rho_T$. Further, it is easy to also see from the constructive proof
of \cite[Theorem 3]{CheGeoPav14b} that steering can be effected with a finite energy control. Thus, from now on, we denote by
 $\mathcal U$ the (non-empty) family of admissible control processes $u$, that is, processes that are adapted\footnote{depend on $t$ and on $\{y^u(s)\mid  0\le s\le t\}$ for each $t\in [0,T]$} to the output process, have finite energy, and effect the steering of \eqref{eq:dynamics} from $\rho_0$ to $\rho_T$.
With the benefit of having resolved the controllability question, we now focus on minimum-energy steering, namely the following.

\begin{problem}\label{formalization}  Determine $u^*$ that minimizes
\begin{equation}\label{finiteenergy}
J(u):=\E\left\{\int_0^Tu(t)' u(t) \,dt\right\}<\infty,
\end{equation}
over all $u\in\mathcal U$, i.e., over adapted inputs that steer the system from state-covariance $\Sigma_0$ to $\Sigma_T$.
\end{problem}

Since $\Sigma_T$ is the covariance of $x(T)$ which is already specified, Problem \ref{formalization} is equivalent to minimizing
    \[
        \tilde{J}(u)=\E\left\{\int_0^Tu(t)' u(t) \,dt\right\}+\E \{x(T)'\Pi(T)x(T)\}
    \]
over all $u\in\cU$. This observation allows us to identify the form of an optimal control strategy. The reasoning is as follows.
Without the terminal constraint on $u$ to meet the end-point state density, it is standard that an $y$-adapted, finite-energy control, minimizing $\tilde J$ is of the form
    \begin{equation}\label{eq:feedback}
        u(t)=-B'\Pi(t)\hat{x}(t),
    \end{equation}
where $\hat{x}(t)$ is the Kalman estimation of $x(t)$ and $\Pi(t)$ satisfies the Riccati equation
\begin{equation}\label{eq:riccati}
\dot{\Pi}(t)=-A(t)'\Pi(t)-\Pi(t)A(t)+\Pi(t)B(t)B(t)'\Pi(t)
\end{equation}
with boundary value $\Pi(T)$ at $t=T$ (see e.g., \cite{GeoLin13}).
Therefore, provided a suitable choice of $\Pi(T)$ can be found so that the controlled state $x^{u^*}(T)$ at $t=T$ has covariance $\Sigma_T$,
this control strategy \eqref{eq:feedback} is the solution to Problem \ref{formalization}. We summarize this conclusion as follows.

\begin{thm}\label{thm:sufficient}
Let $\Pi(\cdot)$ and $P(\cdot)$ be solutions of the Riccati differential equations satisfy \eqref{eq:riccati} and \eqref{eq:Riccati}, respectively, and let $\hat\Sigma(\cdot)$ satisfy
    \begin{eqnarray*}
        \dot{\hat\Sigma}(t)&=&(A-BB'\Pi(t))\hat\Sigma(t)+\hat\Sigma(t)(A-BB'\Pi(t))'
        \\&&\hspace{.2cm}+L(t)DD'L(t)'
    \end{eqnarray*}
with boundary conditions $\Pi(T)$, $P(0)=\Sigma_0$, and $\hat\Sigma(0)=0$.
If $\hat\Sigma(T)=\Sigma_T-P(T)$, then the control law \eqref{eq:feedback} solves Problem \ref{formalization}.
\end{thm}

Theorem \ref{thm:sufficient} provides a sufficient condition for a control $u\in\cU$ to be a solution of Problem \ref{formalization}. The boundary condition $\Pi(T)$ in the statement of the theorem is not specified by the data of the problem; it only needs to be a symmetric matrix and not necessarily positive semi-definite. Thus, the statement of the theorem suggests a shooting method to iterate on the correspondence $\Pi(T)\mapsto \hat\Sigma(T)$ as an approach for obtaining optimal control laws. In general, an optimal solution may not exist and thus, this type of a computation approach that Theorem \ref{thm:sufficient} naturally lends itself to, requires a detail investigation.

However, since for any $\Sigma_T>P(T)$, Problem \ref{formalization} is always feasible and therefore, suboptimal solutions exists (with cost arbitrarily close to $\inf_{u\in\mathcal U}J(u)$). A possible computational approach to construct such suboptimal solutions is given next.

\subsection{Numerical optimization scheme}

We follow steps that are analogous to our recent work \cite{CheGeoPav14b} on minimum-energy steering via state-feedback.
Herein, we consider the control-energy functional
    \begin{eqnarray*}
        J(u)&=&\mE\left\{\int_0^T(K(t)\hat{x}(t))'(K(t)\hat{x}(t))\,dt\right\}\\
        &=& \int_0^T \tr(K(t)(\Sigma(t)-P(t))K(t)')dt\\
        &=& \int_0^T \tr(K(t)\hat{\Sigma}(t)K(t)')dt
    \end{eqnarray*}
to be minimized over $K(t)$ so that \eqref{eq:covdifference} holds as well as the boundary conditions
    \begin{equation}\label{eq:SDPconstraint1}
        \hat{\Sigma}(0)=0,~\mbox{and}~ \hat{\Sigma}(T)=\Sigma(T)-P(T).
    \end{equation}
Let $U(t)=-\hat{\Sigma}(t)K(t)'$.
The objective function becomes\footnote{Note that, as indicated earlier, $\hat{\Sigma}(t)>0$ for $t\in (0,T]$.}
    \[
        J(u)= \int_0^T \tr(U(t)'\hat{\Sigma}(t)^{-1}U(t))dt,
    \]
which is jointly convex in $U(\cdot)$ and $\hat{\Sigma}(\cdot)$. The constraint \eqref{eq:covdifference} also becomes linear in $U$, namely,
    \begin{equation}\label{eq:SDPconstraint2}
        \dot{\hat\Sigma}=A\hat{\Sigma}+\hat{\Sigma} A'+LDD'L'+BU'+UB'.
    \end{equation}
Optimizing $J(u)$ can now be recast as the semi-definite program to minimize
    \[
        \int_0^T \tr(Y(t))dt
    \]
subject to \eqref{eq:SDPconstraint1}-\eqref{eq:SDPconstraint2} and
    \begin{equation*}
    \left[\begin{matrix}Y(t)& U(t)' \\U(t) & \hat{\Sigma}(t)\end{matrix}\right]\ge 0.
    \end{equation*}
    This can now be solved numerically via discretization in time and space.
A (suboptimal) control feedback gain $K(\cdot)$ then can be recovered by $K(t)=-U(t)'\hat{\Sigma}(t)^{-1}$.

\section{Infinite horizon steering}

We now consider the stationary counterpart of our problem to ensure a terminal state-distribution by output feedback.


\subsection{Feasibility and characterization of stationary statistics}

Consider the stationary Kalman filter
    \[
        d\hat{x}(t)=A\hat{x}(t)dt+Bu(t)+L(dy-C\hat{x}dt).
    \]
As usual, the Kalman gain is $L=PC'(DD')^{-1}$ where
$P$ is the covariance of the estimation error $\tilde{x}(t)=x(t)-\hat{x}(t)$
and satisfies the Algebraic Riccati Equation (ARE)
    \begin{equation}\label{eq:ARE}
        AP+PA'+B_1B_1'-PC'(DD')^{-1}CP=0.
    \end{equation}
 It is a direct consequence of optimality of the Kalman filter, just as in \eqref{eq:optimalityofKF} for the finite interval case,
 that for any linear (dynamical and causal) control scheme that ensures stationarity, the covariance $\Sigma$ of the state vector must satisfy
\begin{equation}\label{eq:necessary2}
\Sigma\geq P.
\end{equation}
For any such input, define $S_{ux}:=\mE\{u(t)x(t)'\}$ ($=S_{xu}'$).
Standard It\^o calculus gives
\begin{align*}
d (x(t)x(t)') = &\;(A x(t)x(t)' + x(t)x(t)' A' +B_1B_1')dt\\
&+ (B u(t)x(t)' + x(t)u(t)' B')dt\\
& +
B_1dw(t)x(t)' + x(t)dw(t)' B_1'.
\end{align*}
By taking the expectation we obtain
\begin{align*}
0 = & A\Sigma + \Sigma A'+ B_1B_1'+B S_{ux} +S_{ux}'B',
\end{align*}
and therefore, for any feasible state covariance $\Sigma$,
\begin{subequations}\label{eq:equivalent}
\begin{eqnarray}\label{eq:lyapunov3}
&&A\Sigma + \Sigma A'+B_1B_1'+BX'+XB'=0\\
&&\mbox{can be solved for $X$}.\nonumber
\end{eqnarray}
Condition \eqref{eq:lyapunov3} can be equivalently expressed as:
\begin{equation}\label{eq:rank2}
{\rm rank}\left[\begin{matrix}
A\Sigma+\Sigma A'+B_1B_1' & B\\
B & 0
\end{matrix}\right]
=
{\rm rank}\left[\begin{matrix}
0 & B\\
B & 0
\end{matrix}\right].
\end{equation}
and ensures that $A\Sigma + \Sigma A'+B_1B_1'$ is in the range of the linear map  $X\mapsto BX'+XB'$, cf.\  \cite[Proposition 1]{Geo02a}. See also  \cite{HotSke87} for an alternative but equivalent condition in terms of $A\Sigma + \Sigma A'+B_1B_1'$ belonging to the kernel of a suitable operator.
\end{subequations}
We summarize our conclusion as follows.

\begin{thm}\label{thm:admissiblestate3earlier}
If a positive-definite matrix $\Sigma\geq P$ can be assigned as
the stationary state covariance of \eqref{eq:dynamics} via a suitable choice of feedback control, then $\Sigma$ satisfies any of the equivalent statements (\ref{eq:lyapunov3}-\ref{eq:rank2}).
\end{thm}

We next discuss the converse direction. In this we explain that
the equivalent conditions \eqref{eq:equivalent} together with $\Sigma>P$ are almost sufficient for $\Sigma$ to be a stationary state covariance, in the sense that a covariance matrix arbitrarily close to $\Sigma$ is admissible.
Moreover, we show that this can be achieved by output feedback that is implemented by a Kalman filter and control $u(t)=-K\hat x(t)$.

We begin by considering the joint dynamics of the system state and estimation error $\tilde{x}(t)=x(t)-\hat{x}(t)$, namely,
    \begin{eqnarray*}
        \left[\begin{array}{c}dx\\d\tilde{x}\end{array}\right]&=&
        \left[\begin{array}{cc}A-BK & BK\\ 0 & A-LC\end{array}\right]
        \left[\begin{array}{c}x\\\tilde{x}\end{array}\right]dt\\&&+
        \left[\begin{array}{c}B_1dw\\B_1dw-LDdv\end{array}\right].
    \end{eqnarray*}
The steady-state state covariance of this system satisfies the algebraic Lyapunov equation
    \begin{eqnarray}\nonumber
    0&=&
    \left[\begin{array}{cc}A-BK & BK\\ 0 & A-LC\end{array}\right]
    \left[\begin{array}{cc}\Sigma & P\\P & P\end{array}\right]\\\nonumber
    &&+
    \left[\begin{array}{cc}\Sigma & P\\P & P\end{array}\right]
    \left[\begin{array}{cc}A-BK & BK\\ 0 & A-LC\end{array}\right]'\\\label{eq:biglyap}
    &&+\left[\begin{array}{cc}B_1B_1' & B_1B_1'\\ B_1B_1' & B_1B_1'+LDD'L'\end{array}\right].
    \end{eqnarray}
It follows that
    \begin{subequations}
    \begin{equation}\label{eq:linearcons}
        A\Sigma+\Sigma A'+B_1B_1'-BK(\Sigma-P)-(\Sigma-P)K'B'=0,
    \end{equation}
    and
    \begin{equation}\label{eq:lyapcons}
        (A-BK)(\Sigma-P)+(\Sigma-P)(A-BK)'+LDD'L'=0.
    \end{equation}
    \end{subequations}
    and therefore, $\Sigma$ satisfies \eqref{eq:lyapunov3} for $K=-X'(\Sigma-P)^{-1}$.
    
    Provided $A-BK$ is a Hurwitz matrix, $\Sigma$ is an admissible stationary covariance.
However, in general, $A-BK$ may fail to be Hurwitz because of imaginary eigenvalues. In this case there is a ``nearby'' admissible stationary state-covariance. This can be shown by adapting a similar argument that was used for the case of state-feedback in \cite[Remark 5]{CheGeoPav14b}. Briefly,
let $\hat{\Sigma}=\Sigma-P>0$, and consider the control
\begin{subequations}\label{eq:scheme}
\begin{equation}\label{eq:scheme1}
K_\epsilon = K + \frac12 \epsilon B' \hat{\Sigma}^{-1}
\end{equation}
for $\epsilon>0$. Then, from \eqref{eq:lyapcons},
\begin{eqnarray*}
(A-BK_\epsilon)\hat{\Sigma} + \hat{\Sigma}(A-BK_\epsilon)'&=&-\epsilon BB'-LDD'L'\\
& \leq& -\epsilon BB'.
\end{eqnarray*}
The fact that $A-BK_\epsilon$ is Hurwitz is now obvious. Let $\Sigma_\epsilon$ be the solution to
\begin{equation}\label{eq:scheme2}
(A-BK_\epsilon)(\Sigma_\epsilon-P) + (\Sigma_\epsilon-P) (A-BK_\epsilon)'=-LDD'L'.
\end{equation}
\end{subequations}
Then, the difference $\Delta=\Sigma-\Sigma_\epsilon\geq 0$ satisfies
\begin{eqnarray*}
(A-BK_\epsilon)\Delta + \Delta(A-BK_\epsilon)'&=&-\epsilon BB',
\end{eqnarray*}
and hence is of order $O(\epsilon)$. Thus, the algebraic condition
that $\Sigma$ satisfies \eqref{eq:equivalent} and the positivity constraint $\Sigma>P$ together, are in effect {\em sufficient} in this approximate sense that we just explained since $\Sigma_\epsilon$ is an admissible state covariance.


\subsection{Conditions for optimality}
In general, since there may be more than one solution, we focus on one that minimizes the expected input power (energy rate)
\begin{eqnarray}\label{eq:power}
J_{\rm power}(u)&:=&\E\{u'u\}.
\end{eqnarray}
Thus, assuming feasibility for a specified state covariance $\Sigma$ we consider the following problem.
\begin{problem}\label{problem2}
Determine $u^*$ that minimizes\footnote{Or, equivalently, a $u$ that minimizes
$\lim_{T\to \infty}\frac{1}{T}\E\left\{\int_0^Tu(t)'u(t)dt\right\}$.} $J_{\rm power}(u)$ over all $u(t)=-K\hat{x}(t)$ such \begin{equation}\label{eq:invdensity}
\rho(x)=(2\pi)^{-n/2}\det (\Sigma)^{-1/2}\exp\left(-\frac{1}{2}x'\Sigma^{-1}x\right)
\end{equation}
is the stationary distribution for the state vector.
\end{problem}

Problem \ref{problem2} admits the following finite-dimensional reformulation. Let $\mathcal K$ be the set of all $m\times n$ matrices $K$ such that the corresponding feedback matrix $A-BK$ is Hurwitz. Since $\E\{\hat{x}\hat{x}'\}=\hat{\Sigma}$,
\[
\E\{u'u\}=\E\{\hat{x}'K'K\hat{x}\}=\tr(K\hat{\Sigma} K'),
\]
and Problem \ref{problem2} reduces to finding a $K\in\mathcal K$ which minimizes
\begin{equation}\label{eq:criterion}
J(K)=\tr\left(K\hat{\Sigma} K'\right)
\end{equation}
subject to the constraint \eqref{eq:linearcons}.
Now, consider the Lagrangian
 \begin{eqnarray}
 \mathcal{L}(K,\Pi)&=&\tr\left(K\hat{\Sigma} K'\right)\\\nonumber&&\hspace*{-2.2cm}+\tr\left(\Pi((A-BK)\hat{\Sigma}+\hat{\Sigma}(A'-K'B')+LDD'L')\right).
 \end{eqnarray}
Note that since $\mathcal K$ is {\em open}, a minimum point may fail to exist.
Standard variational analysis leads to the form
$
K=B'\Pi
$
for the optimal gain. This analysis provides the following sufficient condition for optimality.

\begin{prop}\label{prop:prop1}
Assume that there exists a symmetric matrix $\Pi$ such that $A-BB'\Pi$ is a Hurwitz matrix and
\begin{equation}\label{eq:sigmastat}
(A-BB'\Pi)\hat{\Sigma}+\hat{\Sigma}(A-BB'\Pi)'+LDD'L'=0
\end{equation}
holds.
Then,
\begin{equation}\label{eq:statoptcontr}
u^*(t)=-B'\Pi \hat{x}(t)
\end{equation}
is a solution to Problem \ref{problem2}.
\end{prop}

\subsection{Minimum energy control}
In a similar manner as before, we next provide a numerical scheme to compute a solution to Problem \ref{problem2}. The average input power (energy rate) is
\begin{eqnarray}
\E\{u'u\}
&=&\tr(K\hat{\Sigma} K')
\\\nonumber
&=&\tr(X'\hat{\Sigma}^{-1}X)
\end{eqnarray}
in either $K$, or $X$.
Thus, the optimal constant feedback gain $K$ can be obtained by solving the convex optimization problem
\begin{equation}\label{eq:XSX}
\min\left\{\tr(X'\Sigma^{-1}X)\mid \mbox{\eqref{eq:lyapunov3} holds } \right\}.
\end{equation}
After obtaining the optimal $K$ from \eqref{eq:XSX}, we need to check whether $A-BK$ is Hurwitz or not. If not, the scheme in \eqref{eq:scheme} can be applied to approximate the solution.

\begin{figure}[htb]\begin{center}
\includegraphics[width=0.49\textwidth]{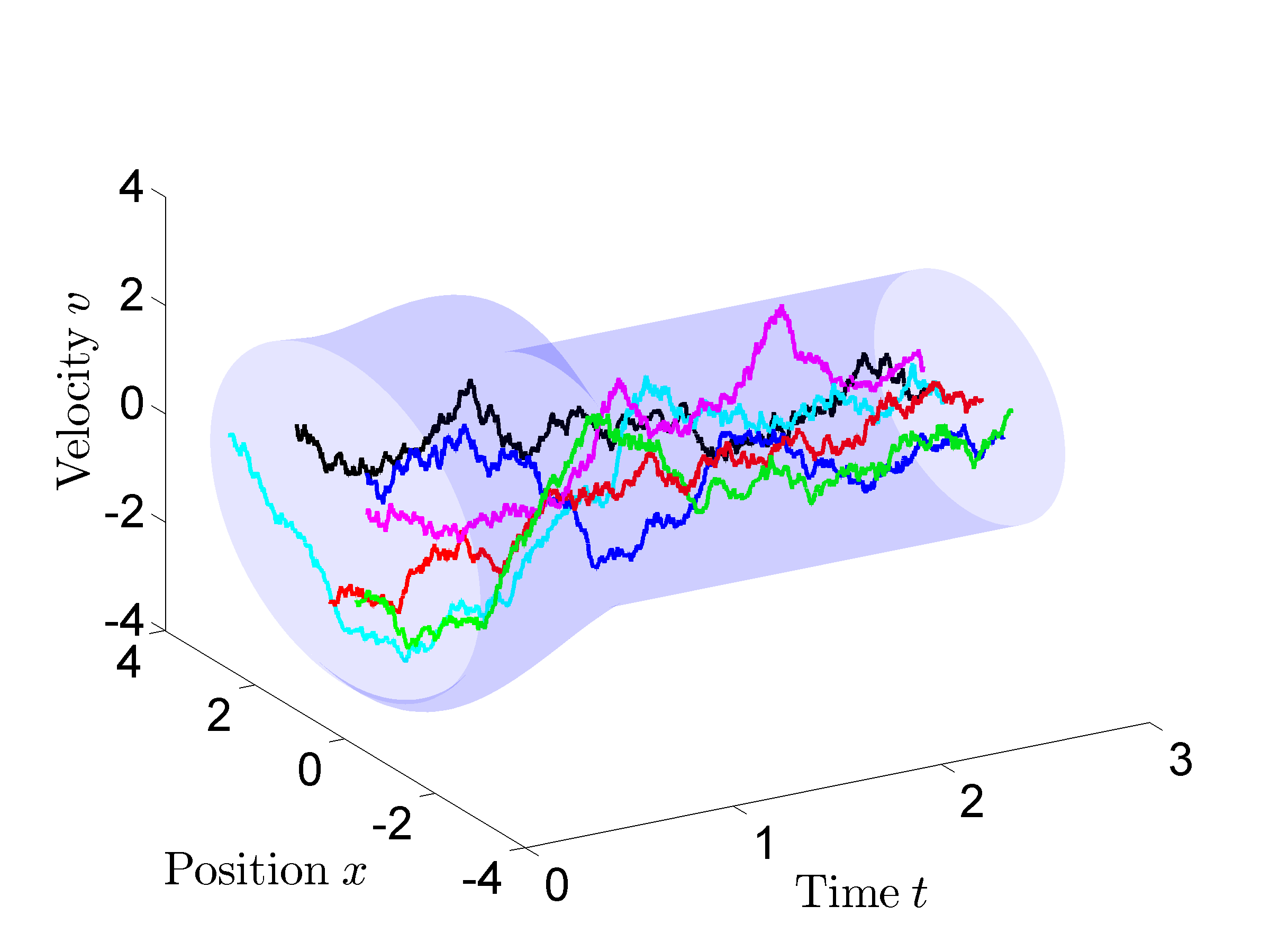}
   \caption{State trajectories in phase space}
   \label{fig:Eg1Phase1}
\end{center}\end{figure}
\begin{figure}[htb]\begin{center}
\includegraphics[width=0.49\textwidth]{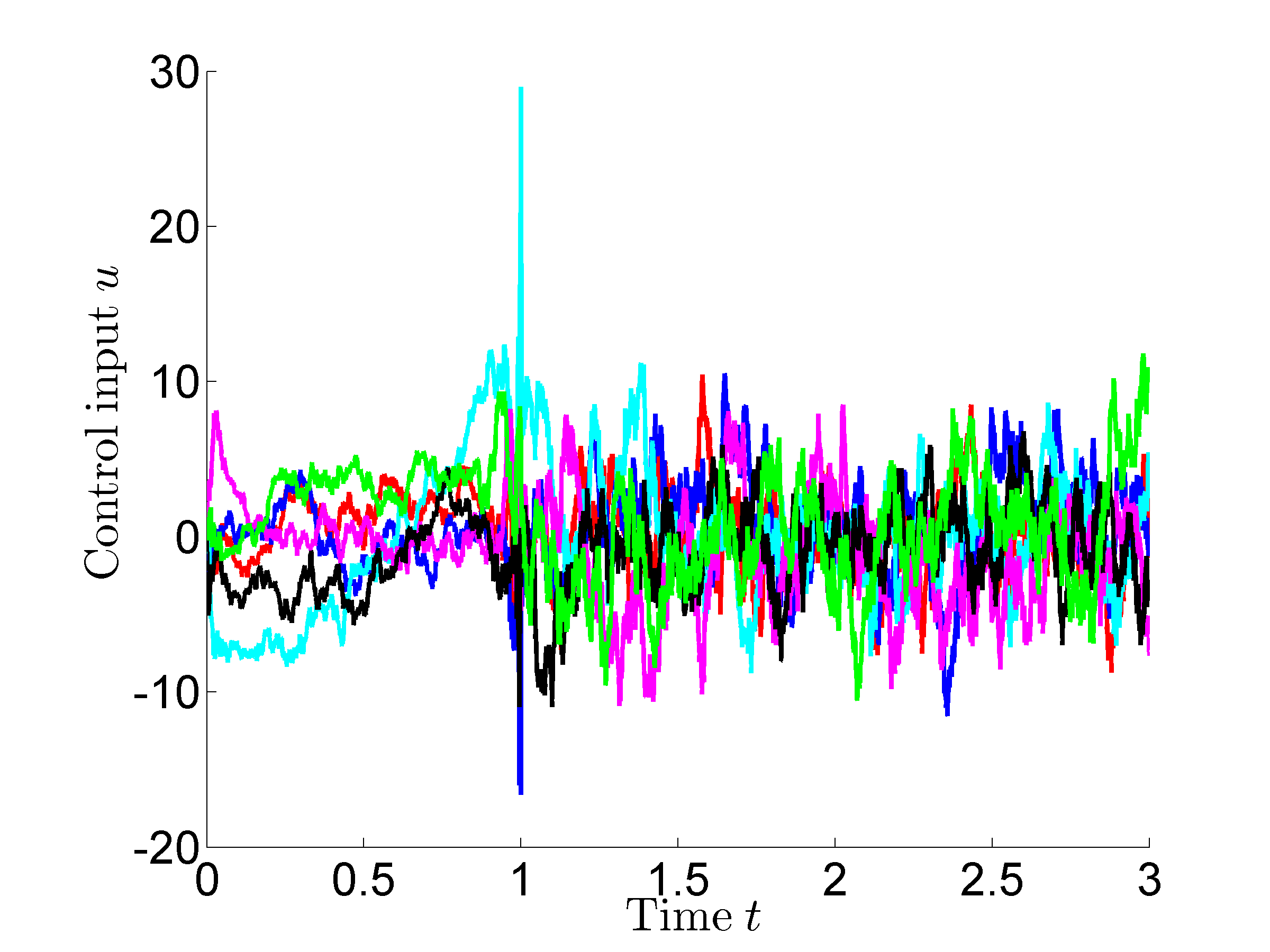}
   \caption{Control input}
   \label{fig:Eg1Control1}
\end{center}\end{figure}
\section{Numerical example}
We consider motion of particles modeled by
   \begin{eqnarray*}
       dx_1(t) &=& x_2(t)dt\\
       dx_2(t) &=& u(t)dt+dw(t)\\
       dy(t) &=& x_1(t)dt+0.1dv(t)
   \end{eqnarray*}
Here, $u(t)$ represents a control input (force) at our disposal, $x_1(t)$ represents position, $x_2(t)$ velocity, $y(t)$ noisy (integral) position measurements, while $dw(t)$ models random white-noise forcing and $dv(t)$ represents measurement noise.
Our goal is to steer the spread of the particles from an initial Gaussian distribution with $\Sigma_0=I$ at $t=0$ to a terminal marginal $\Sigma_1=\frac{1}{2}I$ at $t=1$, and to maintain the particles' distribution constant after $t=1$. For the steering part on the interval $[0,\,1]$, $\Sigma_1$ is feasible since $\Sigma_1>P(1)$, where
    \[
        P(1)=\left[\begin{matrix} 0.0471 & 0.1049\\ 0.1049 & 0.4587\end{matrix}\right]
    \] 
is the estimation error of Kalman filter at $t=1$. For the maintaining part, it is possible since $\Sigma_1$ satisfies \eqref{eq:lyapunov3} with $X=[-1/2,\,-1/2]'$, and $\Sigma_1$ is greater than the solution of the Algebraic Riccati Equation \eqref{eq:ARE}, which is,
    \[
        P=\left[\begin{matrix} 0.0447 & 0.1000\\ 0.1000 & 0.4472\end{matrix}\right].
    \]
Moreover, the corresponding feedback gain $K=[5.4440,\,19.7854]$ makes $A-BK$ be Hurwitz. 

We now implement the time-varying output feedback as explained in Section \ref{sec:finitehorizon} to steer the distribution of particles over the interval $[0,\,1]$ and from there on, we implement the stationary control consisting of the stationary Kalman filter and the above constant gain.
Figure~\ref{fig:Eg1Phase1} displays typical sample paths in phase space, as functions of time, over the time window $[0,\,3]$, and Figure \ref{fig:Eg1Control1} displays the corresponding (color coded) control signals, $u(t)=-K(t)\hat{x}(t)$.

\section{Concluding remarks}

We have addressed the problem of steering the state statistics of a linear stochastic system via output feedback.
In this case, where only partial state observation is available, we have provided
necessary and sufficient conditions for being able to specify a terminal Gaussian distribution for the state vector as well as
a stationary Gaussian distribution. The paper builds on our recent work \cite{CheGeoPav14a,CheGeoPav14b} where we studied the problem to steer state statistics via state feedback. The viewpoint presented herein differs from standard Linear Quadratic Regulator theory \cite{FR,Astrom} in that the control objective is specified directly in terms of terminal or stationary distributions for the state vector. Applications of this viewpoint are envisioned in areas where a distribution rather than a set of values for the state vector is a natural specification, e.g., in quality control, industrial and manufacturing processes, as well as in thermally driven atomic force microscopy,
the control of molecular motors, laser driven reactions, manipulation of macromolecules, and so on, see e.g. \cite{toyabe2010nonequilibrium,gannepalli2005thermally,braiman2003control,hayes2001active,ricci2014low}. Future work is expected to focus on such applications based on this framework.

\spacingset{.95}
\bibliographystyle{IEEEtran}
\bibliography{../refs}
\end{document}